\documentclass{amsart}

\usepackage{amsmath}
\usepackage{amsthm}
\usepackage{amssymb}
\usepackage{amsfonts}
\usepackage{enumitem}
\usepackage{setspace}
\usepackage[bookmarks=true,hyperindex, pdftex,colorlinks,citecolor=blue]{hyperref}
\usepackage{centernot} %for negations
\usepackage{marginnote} % add margin notes
\usepackage[left=3cm,right=3cm,top=2cm,bottom=2cm]{geometry}
\usepackage{bbm}
\usepackage[all]{xy}
\usepackage{tikz}
\usetikzlibrary{arrows}
\usetikzlibrary{shapes,decorations}
\usetikzlibrary{positioning}
\def\<{\langle}
\def\>{\rangle}

\theoremstyle{plain}
\newtheorem{theorem}{Theorem}[section]

\newtheorem{proposition}[theorem]{Proposition}

\theoremstyle{definition}
\newtheorem{definition}[theorem]{Definition}
\newtheorem{example}[theorem]{Example}

\theoremstyle{remark}
\newtheorem{remark}[theorem]{Remark}
\newtheorem{remarks}[theorem]{Remarks}
\numberwithin{equation}{section}
\usetikzlibrary{matrix}
\usepackage[all]{xy}
\begin{document}
%%%%%%%%%%%%%%%%%%%%%%%%%%

\title[On Super-Recurrent Operators]{On Super-Recurrent Operators}

\author[M. Amouch]{Mohamed Amouch}

\author[O. Benchiheb]{Otmane Benchiheb}

\address[Mohamed Amouch and Otmane Benchiheb]{University Chouaib Doukkali.
Department of Mathematics, Faculty of science
Eljadida, Morocco} 
\email{amouch.m@ucd.ac.ma}
\email{otmane.benchiheb@gmail.com}

\subjclass[2010]{
47A16, %Cyclic vectors, hypercyclic and chaotic operators
    37B20  	%Notions of recurrence and recurrent behavior in dynamical systems
%	%37A45  %	Relations between ergodic theory and number theory
%	%37A25  %	Ergodicity, mixing, rates of mixing
%		%11B25 % 	Arithmetic progressions
%		%47B37  	%Linear operators on special spaces (weighted shifts, operators on sequence spaces, etc.)
%	%47H60,  Multilinear and polynomial operators
%	%37F10, Complex dynamical systems - Polynomials; rational maps; entire and meromorphic functions
%	%30D20, %Functions of a complex variable, Entire functions, general theory 
%	%30K99  % Functions of a complex variable -Universal holomorphic functions -None of the above, but in this section
%	%46E50 %  	Spaces of differentiable or holomorphic functions on infinite-dimensional spaces
%	%46T25%    	Functional analysis	Holomorphic maps
%	%32H50.  %Several complex variables and analytic spaces	Iteration problems
%	%32A19 % 	Holomorphic functions of several complex variables  	Normal families of functions, mappings
%%15A69 %Multilinear algebra, tensor products,
%%46G25% 	(Spaces of) multilinear mappings, polynomials 
%%46G20% 	Infinite-dimensional holomorphy
%}
}
%\keywords{}

\date{} % uncomment to remove date from title

% ABSTRACT

\begin{abstract}
In this paper, we introduce and study the notion of super-recurrence of operators.
We investigate some properties of this class of operators and show that it shares some characteristics with supercyclic and recurrent operators.
In particular, we show that if $T$ is super-recurrent, then $\sigma(T)$ and $\sigma_p(T^*)$,
the spectrum of $T$ and the point spectrum of $T^*$ respectively,
have some noteworthy properties.
\end{abstract}

%-----------------------------------

\maketitle
%%%%%%%%%%%%%%%%%%%%%%%%%%%%%%%%%%%%%%%%%%%%%%%%%%%%%%%%%%%%%%%%%%%%%%%%%%%%%%%%%%%%%%%%%%%%%%%%%%%%%%%%%
\section{Introduction and preliminaries}
Throughout this paper, $X$ will denote a Banach space over the field $\mathbb{C}$ of complex numbers.
By an operator, we mean a linear and continuous map acting on $X$.

The most important and studied notions in the linear dynamical system are those of hypercyclicity and supercyclicity$:$

An operator $T$ acting on $X$ is said to be hypercyclic if there exists a vector $x$ whose orbit under $T$; $Orb(T,x):=\{T^n x \mbox{ : }n\in \mathbb{N}\}$, is dense in $X.$
The vector $x$ is called a hypercyclic vector for $T$.
The set of all hypercyclic vectors for $T$ is denoted by $HC(T).$
One of the first examples of hypercyclic operators on the Banach space setting was given in $1969$ by Rolewicz \cite{Rolewicz}.
%%%%%%%%%%%%%%%%%%%%%%%%%%%%%%%%%%%%%%%%%%%%%%%%%%%%%%%%%%%%%%%%%%%%%%%%%%%%%%%%%%%%%%%%%%%%%%%%%%%%%%%%%

Birkhoff introduced an equivalent notion of the hypercyclicity called topological transitivity$:$
an operator $T$ acting on a separable Banach space is hypercyclic if and only if it is topologically transitive, that is, for each pair $(U,V)$ of nonempty open subsets of $X$ there exists some positive integer $n$ such that $T^n(U)\cap V\neq\emptyset$, see \cite{Birkhoff}.

%%%%%%%%%%%%%%%%%%%%%%%%%%%%%%%%%%%%%%%%%%%%%%%%%%%%%%%%%%%%%%%%%%%%%%%%%%%%%%%%%%%%%%%%%%%%%%%%%%%%%%%%%

In $1974$, Hilden and Wallen in \cite{HW} introduced the concept of supercyclicity.
An operator $T$ acting on $X$ is said to be supercyclic if there exists some vector $x$ whose scaled orbit under $T$; $\mathbb{C}Orb(T,x):=\{\lambda T^n x\mbox{ : }\lambda\in\mathbb{C}\mbox{, }n\in\mathbb{N}\}$, is dense in $X$.
Such a vector $x$ is called a supercyclic vector for $T$.
The set of all supercyclic vectors for $T$ is denoted by $SC(T).$
As in the case of the hypercyclicity, there exists a characterization of the supercyclicity basing on the open subsets of $X.$
An operator $T$ acting on a separable Banach space is supercyclic if and only if for each pair
$(U,V)$ of nonempty open subsets of $X$ there exist $\lambda\in\mathbb{C}$ and $n\in\mathbb{N}$ such that $\lambda T^n(U)\cap V\neq\emptyset.$
%%%%%%%%%%%%%%%%%%%%%%%%%%%%%%%%%%%%%%%%%%%%%%%%%%%%%%%%%%%%%%%%%%%%%%%%%%%%%%%%%%%%%%%%%%%%%%%%%%%%%%%%%

For more information about hypercyclic and supercyclic operators and their proprieties, see the book \cite{Peris} by KG. Grosse-Erdmann and A. Peris , the book \cite{BM} by F. Bayart and E. Matheron, and the survy article \cite{Grosse} by KG. Grosse-Erdmann.
%%%%%%%%%%%%%%%%%%%%%%%%%%%%%%%%%%%%%%%%%%%%%%%%%%%%%%%%%%%%%%%%%%%%%%%%%%%%%%%%%%%%%%%%%%%%%%%%%%%%%%%%%

Another notion in the dynamical system that has a long story is that of recurrence which is introduced by Poincar\'{e} in \cite{Poincare}.
A systematic
study of recurrent operators goes back to the work of Gottschalk and Hedlund \cite{GH} and also the work of Furstenberg \cite{Furstenberg}.
Recently, recurrent operators have been studied in \cite{CMP}.
%%%%%%%%%%%%%%%%%%%%%%%%%%%%%%%%%%%%%%%%%%%%%%%%%%%%%%%%%%%%%%%%%%%%%%%%%%%%%%%%%%%%%%%%%%%%%%%%%%%%%%%%%

An operator $T$ acting on $X$ is said to be recurrent if for each open subset $U$ of $X$,
there exists some positive integer $n$ such that $T^n(U)\cap U\neq\emptyset$.
A vector $x\in X$ is called a recurrent vector for $T$ if there exists an increasing sequence $(n_k)$ of positive integers such that $T^{n_k}x\longrightarrow x$ as $k\longrightarrow\infty.$
The set of all recurrent vectors for $T$ is denoted by $Rec(T)$,
and we have that $T$ is recurrent if and only if $Rec(T)$ is dense in $X.$
For more information about this classe of operators, see \cite{Ak,BGLP,CP,GMOP,YW,HZ,GMQ,GM}.
%%%%%%%%%%%%%%%%%%%%%%%%%%%%%%%%%%%%%%%%%%%%%%%%%%%%%%%%%%%%%%%%%%%%%%%%%%%%%%%%%%%%%%%%%%%%%%%%%%%%%%%%%

%%%%%%%%%%%%%%%%%%%%%%%%%%%%%%%%%%%%%%%%%%%%%%%%%%%%%%%%%%%%%%%%%%%%%%%%%%%%%%%%%%%%%%%%%%%%%%%%%%%%%%%%%

Motivated by the relationship between hypercyclic and recurrent operators, we introduce in this paper a new class of operators called super-recurrent operators which is related to the supercyclicity and recurrence.

In section $2$, we introduce the notion of
super-recurrence for operators. We show that every recurrent operator is super-recurrent
but the converse is false. We also prove that every supercyclic operator is super-recurrent and that there exists
an operator which is super-recurrent but not supercyclic.
In section $3$, 
we prove some proprieties for super-recurrent operators, we prove that 
if $T\in \mathcal{B}(X)$ admits a super-recurrent vector, then it admits an invariant subspace consisting 
except for zero, of super-recurrent vectors. 
Also, we prove that $T$ is super-recurrent if and only if $T$ admits a dense subset of super-recurrent vectors.
Moreover, we prove that $T$ is super-recurrent if and only if $T^p$ is super-recurrent, for every nonzero positive integer $p$.

In section $4$, we focus on the spectral proprieties of super-recurrent operators. We prove that if $T$ is super-recurrent, then $\sigma_p(T^*)$ and $\sigma(T)$ have almost the same proprieties as supercyclic operators.
In particular, we show that  there exists $R>0$ such that each connected component of the spectrum of $T$ intersect the circle $\{z\in\mathbb{C} \mbox{ : }\vert z \vert=R\}$.
Moreover, we prove that the $\sigma_p(T^*)$ is completely contained in a circle of center $0$.
Finally, we show that if $\lambda\in \sigma_p(T^*)$, then one can find a $T$-invariant hyperplane $X_0$
such that $\lambda^{-1}T_{/X_0}$ is recurrent on $X_0$.

\section{Super-recurrent operators}
\begin{definition}
We say that an operator $T$ is super-recurrent if, for every nonempty open subset $U$ of $X$
there exists some $n\geq1$ and some $\lambda\in\mathbb{C}$ such that
$$ \lambda T^{n}(U)\cap U \neq\emptyset. $$ 

A vector $x\in X\setminus\{0\}$ is called a super-recurrent vector for $T$ if there exist
a strictly increasing sequence of positive integers $(k_n)_{n\in\mathbb{N}}$ and a sequence $(\lambda_{k_n})_{n\in\mathbb{N}}$ of complex numbers such that
$$ \lambda_{k_n} T^{k_n}x\longrightarrow x $$
as $n\longrightarrow+\infty$.
We will denote by $SRec(T)$ the set of all super-recurrent vectors for $T.$
\end{definition}
%%%%%%%%%%%%%%%%%%%%%%%%%%%%%%%%%%%%%%%%%%%%%%%%%%%%%%%%%%%%%%%%%%%%%%%%%%%%%%%%%%%%%%%%%%%%%%%%%%%%%%%%%
\begin{remarks}\label{rem}

\begin{enumerate}
\item The supercyclicity implies the super-recurrence.
However, the converse does not hold in general. Indeed, let $n\in\mathbb{N}$ and $\lambda_1,\dots,\lambda_n$ be nonzero  complex numbers such that $\vert\lambda_i\vert=\vert\lambda_j\vert=R$ for some strictly positive real number $R$, for $1\leq i \mbox{, }j\leq n$.
We define an operator $T$ on $\mathbb{C}^n$ by
$$\begin{array}{ccccc}
T & : & \mathbb{C}^n & \longrightarrow & \mathbb{C}^n \\
 & & (x_1,\dots,x_n) & \longmapsto & (\lambda_1 x_1,\dots,\lambda_n x_n). \\
\end{array}$$
Let $U$ be a nonempty open subset of $X$ and $x\in U$.
Since $\vert R^{-1}\lambda_i\vert=1 $, for all $1\leq i\leq n$, it follows that there exists a strictly increasing sequence of positive integers $(k_n)_{n\in\mathbb{N}}$ such that $\left( R^{-1}\lambda_i\right)^{k_n}\longrightarrow 1$, for all $1\leq i\leq n$.
Let $\lambda_k=R^{-k_n}$, for all $k$, then
$$ \lambda_k T^{k_n}x \longrightarrow x. $$
as $k\longrightarrow \infty.$
Since $x\in U$ and $U$ is an open subset of $X$, it follows that there exists $k_0$ such that $\lambda_{k_0}T^{n_{k_0}}x\in U$.
Hence
$$ \lambda_{k_0}T^{n_{k_0}}(U)\cap U\neq\emptyset. $$
This means that $T$ is a super-recurrent operators.
However, $T$ cannot be supercyclic whenever $n\geq2$,
since a Banach space $X$ supports supercyclic operators if and only if dim$(X)=1$ or dim$(X)=\infty$, see \cite{HW}.

\item A recurrent operator is super-recurrent, but the converse does not hold in general.
Indeed, if $T$ is the operator defned in $(1)$, then $T$ is recurrent if and only if $\vert \lambda_i\vert=1$, for all $1\leq i\leq n$, see \cite{CMP}.
\end{enumerate} 
\end{remarks}
%%%%%%%%%%%%%%%%%%%%%%%%%%%%%%%%%%%%%%%%%%%%%%%%%%%%%%%%%%%%%%%%%%%%%%%%%%%%%%%%%%%%%%%%%%%%%%%%%%%%%%%%%
We have the following diagram showing the relationships among super-recurrence, recurrence and supercyclicity. 
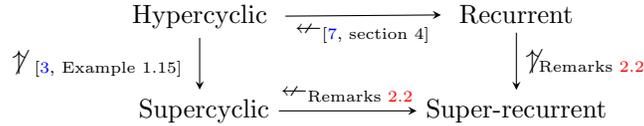
\begin{center}
\begin{tikzpicture}
  \matrix (m) [matrix of math nodes,row sep=2em,column sep=5em,minimum width=5em]
  {
    \mbox{ Hypercyclic }&\mbox{ Recurrent }\\
    \mbox{ Supercyclic} & \mbox{Super-recurrent} \\};
  \path[-stealth]
    (m-1-1) edge  node [left] {$\not\uparrow_{\mbox{ \tiny{\cite[Example 1.15]{BM} }}}$} (m-2-1)
            edge  node [below] {$\nleftarrow_{\mbox{\tiny{\cite[section 4]{CMP}}}}$} (m-1-2)
    (m-2-1.east|-m-2-2) edge  node [below] {}
            node [above] {$\nleftarrow_{\mbox{\tiny{Remarks \ref{rem}}}}$} (m-2-2)
    (m-1-2) edge  node [right] {$\not\uparrow_{\mbox{\tiny{Remarks \ref{rem}}}}$} (m-2-2);
            %edge [dashed,-] (m-2-1);
\end{tikzpicture}
\end{center}

%%%%%%%%%%%%%%%%%%%%%%%%%%%%%%%%%%%%%%%%%%%%%%%%%%%%%%%%%%%%%%%%%%%%%%%%%%%%%%%%%%%%%%%%%%%%%%%%%%%%%%%%%
\section{Some properties of super-recurrent operators}
%%%%%%%%%%%%%%%%%%%%%%%%%%%%%%%%%%%%%%%%%%%%%%%%%%%%%%%%%%%%%%%%%%%%%%%%%%%%%%%%%%%%%%%%%%%%%%%%%%%%%%%%%

In the following, we give some properties satisfies by super-recurrent operators.

\begin{proposition}\label{pr11}
If $S\in\mathcal{B}(X)$ is an operator such that $TS=ST$, then $SRec(T)$ is invariant under $S.$
\end{proposition}
%%%%%%%%%%%%%%%%%%%%%%%%%%%%%%%%%%%%%%%%%%%%%%%%%%%%%%%%%%%%%%%%%%%%%%%%%%%%%%%%%%%%%%%%%%%%%%%%%%%%%%%%%%
\begin{proof}
Let $x\in SRec(T)$.
Then there exist a strictly increasing sequence of positive integers $(k_n)_{n\in\mathbb{N}}$ and a sequence $(\lambda_{k_n})_{n\in\mathbb{N}}$ of complex numbers such that
$ \lambda_{k_n} T^{k_n}x\longrightarrow x $
as $n\longrightarrow+\infty$.
Since $S$ is continuous and $TS=ST$, it follows that
$ \lambda_{k_n} T^{k_n}Sx\longrightarrow Sx $
 as $n\longrightarrow+\infty$.
 This means that $Sx\in SRec(T)$.
\end{proof}
%%%%%%%%%%%%%%%%%%%%%%%%%%%%%%%%%%%%%%%%%%%%%%%%%%%%%%%%%%%%%%%%%%%%%%%%%%%%%%%%%%%%%%%%%%%%%%%%%%%%%%%%%%
We are now ready to deduce an important result on the algebraic structure of the set of super-recurrent vectors.

Recall that if $p(z)=\sum_{i=0}^n\lambda_i z^i$ and $T\in\mathcal{B}(X)$, then $p(T)=\sum_{i=0}^n\lambda_i T^i.$
%%%%%%%%%%%%%%%%%%%%%%%%%%%%%%%%%%%%%%%%%%%%%%%%%%%%%%%%%%%%%%%%%%%%%%%%%%%%%%%%%%%%%%%%%%%%%%%%%%%%%%%%%%
\begin{theorem}
If $x$ is a super-recurrent vector for T, then
$$\{p(T)x \mbox{ :  p is a polynomial} \}\setminus\{0\}\subset SRec(T).$$
In particular, If $T$ has a super-recurrent vector, then it admits an invariant subspace
consisting, except for zero, of super-recurrent vectors.
\end{theorem}
%%%%%%%%%%%%%%%%%%%%%%%%%%%%%%%%%%%%%%%%%%%%%%%%%%%%%%%%%%%%%%%%%%%%%%%%%%%%%%%%%%%%%%%%%%%%%%%%%%%%%%%%%
\begin{proof}
For a nonzero polynomial $p$, let $S=p(T)$. Then $ST=TS.$ 
Since $x\in SRec(T)$, it follows by Proposition \ref{pr11}, that $p(T)x\in SRec(T).$
\end{proof}
%%%%%%%%%%%%%%%%%%%%%%%%%%%%%%%%%%%%%%%%%%%%%%%%%%%%%%%%%%%%%%%%%%%%%%%%%%%%%%%%%%%%%%%%%%%%%%%%%%%%%%%%%
\begin{remark}\label{rem dense}
If $T$ is a super-recurrent operator, then it is of dense range.
\end{remark}
%%%%%%%%%%%%%%%%%%%%%%%%%%%%%%%%%%%%%%%%%%%%%%%%%%%%%%%%%%%%%%%%%%%%%%%%%%%%%%%%%%%%%%%%%%%%%%%%%%%%%%%%%
Let $X$ and $Y$ be two Banach spaces.
If $T$ and $S$ are operators acting on $X$ and $Y$ respectively, then $T$ and $S$ are called quasi-conjugate or quasi-similar if there exists some operator $\phi$ : $X\longrightarrow Y$ with dense range such $S\circ\phi=\phi\circ T.$
If $\phi$ can be chosen to be a homeomorphism, then $T$ and $S$ are called conjugate or similar, see \cite[Definition 1.5]{Peris}.

%%%%%%%%%%%%%%%%%%%%%%%%%%%%%%%%%%%%%%%%%%%%%%%%%%%%%%%%%%%%%%%%%%%%%%%%%%%%%%%%%%%%%%%%%%%%%%%%%%%%%%%%%
\begin{proposition}
Assume that $T\in\mathcal{B}(X)$ and $S\in\mathcal{B}(Y)$ are quasi-similar.
Then, $T$ is super-recurrent in $X$ implies that $S$ is super-recurrent in $Y$.
\end{proposition}
%%%%%%%%%%%%%%%%%%%%%%%%%%%%%%%%%%%%%%%%%%%%%%%%%%%%%%%%%%%%%%%%%%%%%%%%%%%%%%%%%%%%%%%%%%%%%%%%%%%%%%%%%
\begin{proof}
Suppose that $T$ is super-recurrent.
If $U$ is a nonempty open subset of $Y$,
then $\phi^{-1}(U)$ is a nonempty open subset of $X$.
Since $T$ is super-recurrent,
it follows that there exist $n\in\mathbb{N}$, $\lambda \in\mathbb{C}$ and $x\in X$ such that $x\in \phi^{-1}(U)$ and $\lambda T^{n}x\in \phi^{-1}(U)$, this means that $\phi(x)\in U$ and $\lambda \phi\circ T^{n}(x)\in U$.
Since $T$ and $S$ are quasi-similar,
it follows that $\phi(x)\in U$ and $\lambda S^{n}\circ\phi(x)\in U$.
Hence, $S$ is super-recurrent in $Y$.
\end{proof}
%%%%%%%%%%%%%%%%%%%%%%%%%%%%%%%%%%%%%%%%%%%%%%%%%%%%%%%%%%%%%%%%%%%%%%%%%%%%%%%%%%%%%%%%%%%%%%%%%%%%%%%%%%

\begin{remark}
Assume that $T\in\mathcal{B}(X)$ and $S\in\mathcal{B}(Y)$ are similar.
Then, $T$ is super-recurrent in $X$ if and only if $S$ is super-recurrent in $Y$.
\end{remark}

%%%%%%%%%%%%%%%%%%%%%%%%%%%%%%%%%%%%%%%%%%%%%%%%%%%%%%%%%%%%%%%%%%%%%%%%%%%%%%%%%%%%%%%%%%%%%%%%%%%%%%%%%

The following theorem gives necessary and sufficient conditions of super-recurrence of operators.

%%%%%%%%%%%%%%%%%%%%%%%%%%%%%%%%%%%%%%%%%%%%%%%%%%%%%%%%%%%%%%%%%%%%%%%%%%%%%%%%%%%%%%%%%%%%%%%%%%%%%%%%%%
\begin{theorem}\label{tt}
The following assertions are equivalent$:$
\begin{enumerate}
\item  $T$ is super-recurrent;
\item  for each $x\in X,$ there exist a sequence $(n_k)$ of positive integers,
a sequence $(x_{n_k})$ of elements of $X$ and a sequence $(\lambda_{n_k})$ of nonzero complex numbers such that
$$x_{n_k}\longrightarrow x\hspace{0.2cm}\mbox{ and }\hspace{0.2cm}\lambda_{n_k} T^{n_k}( x_{n_k})\longrightarrow x;$$
\item  for each $x\in X$ and for $W$ a neighborhood of zero, there exist $z\in X$, $\lambda\in\mathbb{C}$,
and $n\in\mathbb{N}$  such that
$$\lambda T^n( z)-x\in W\hspace{0.2cm}\mbox{ and }\hspace{0.2cm} z-x\in W.$$
\end{enumerate}
\end{theorem}
%%%%%%%%%%%%%%%%%%%%%%%%%%%%%%%%%%%%%%%%%%%%%%%%%%%%%%%%%%%%%%%%%%%%%%%%%%%%%%%%%%%%%%%%%%%%%%%%%%%%%%%%%%
\begin{proof}$(1)\Rightarrow(2)$
Let $x\in X$. For all $k\geq1$, let $U_k=B(x,\frac{1}{k})$.
Then $U_k$ is a nonempty open subset of $X$.
Since $T$ is super-recurrent,
there exist $n_k\in\mathbb{N}$ and $\lambda_{n_k}$ such that $\lambda_{n_k} T^{n_k}(U_k)\cap U_k\neq\emptyset$.
For all $k\geq1$, let $x_{n_k}\in U_k$ such that $\lambda_{n_k} T^{n_k}( x_{n_k})\in U_k$,
then
$\Vert x_{n_k}-x \Vert<\frac{1}{k}$ and $\Vert \lambda_{n_k} T^{n_k}(x_{n_k})-x \Vert<\frac{1}{k}$
which implies that
$x_{n_k}\longrightarrow x$ and $\lambda_{n_k} T^{n_k}(x_{n_k})\longrightarrow x.$
%%%%%%%%%%%%%%%%%%%%%%%%%%%%%%%%%%%%%%%%%%%%%%%%%%%%%%%%%%%%%%%%%%%%%%%%%%%%%%%%%%%%%%%%%%%%%%%%%%%%%%%%%%

$(2)\Rightarrow(3)$ : It is clear;
%%%%%%%%%%%%%%%%%%%%%%%%%%%%%%%%%%%%%%%%%%%%%%%%%%%%%%%%%%%%%%%%%%%%%%%%%%%%%%%%%%%%%%%%%%%%%%%%%%%%%%%%%%

$(3)\Rightarrow(1)$ Let $U$ be a nonempty open subsets of $X$ and $x\in U$.
Since for all $k\geq1$,  $W_k=B(0,\frac{1}{k})$ is a neighborhood of zero,  there exist $z_k\in X$, $n_k\in\mathbb{N}$ and $\lambda_{n_k}\in\mathbb{C}$ such that
$\Vert \lambda_{n_k} T^{n_k}( z_k)-x\Vert<\frac{1}{k}$ and $\Vert x-z_k\Vert<\frac{1}{k}.$
This implies that
$z_k\longrightarrow x$ and $\lambda_{n_k}T^{n_k}(z_k)\longrightarrow x,$
which implies the result.
 \end{proof}
%%%%%%%%%%%%%%%%%%%%%%%%%%%%%%%%%%%%%%%%%%%%%%%%%%%%%%%%%%%%%%%%%%%%%%%%%%%%%%%%%%%%%%%%%%%%%%%%%%%%%%%%%

%%%%%%%%%%%%%%%%%%%%%%%%%%%%%%%%%%%%%%%%%%%%%%%%%%%%%%%%%%%%%%%%%%%%%%%%%%%%%%%%%%%%%%%%%%%%%%%%%%%%%%%%%
\begin{proposition}\label{rbi}
Assume that $T\oplus S$ is super-recurrent in $X\oplus Y$.
Then $T$ and $S$ are super-recurrent on $X$ and $Y$ respectively.
\end{proposition}
%%%%%%%%%%%%%%%%%%%%%%%%%%%%%%%%%%%%%%%%%%%%%%%%%%%%%%%%%%%%%%%%%%%%%%%%%%%%%%%%%%%%%%%%%%%%%%%%%%%%%%%%%%
\begin{proof}
If $U_1$ and $U_2$ are nonempty open set of $X$ and $Y$ respectively,
then $U_1\oplus U_2$ is a nonempty open set of $X\oplus Y$.
Since $T\oplus S$ is super-recurrent, there exist $n\in\mathbb{N}$ and $\lambda\in\mathbb{C}$ such that
$ (\lambda T^n\oplus S^n)(U_1\oplus U_2)\cap (U_1\oplus U_2)\neq\emptyset, $
which means that $\lambda T^n(U_1)\cap U_1\neq\emptyset$ and $\lambda S^n(U_2)\cap U_2\neq\emptyset$.
Hence $T$ and $S$ are super-recurrent.
\end{proof}

%%%%%%%%%%%%%%%%%%%%%%%%%%%%%%%%%%%%%%%%%%%%%%%%%%%%%%%%%%%%%%%%%%%%%%%%%%%%%%%%%%%%%%%%%%%%%%%%%%%%%%%%%

The next theorem gives the relationship between super-recurrent vectors and super-recurrent operators.
%%%%%%%%%%%%%%%%%%%%%%%%%%%%%%%%%%%%%%%%%%%%%%%%%%%%%%%%%%%%%%%%%%%%%%%%%%%%%%%%%%%%%%%%%%%%%%%%%%%%%%%%%
\begin{theorem}\label{th1}
Let $T$ be an operator acting on $X$.
The following assertion are equivalent$:$
\begin{itemize}
\item[$(1)$] $T$ admits a dense subset of super-recurrent vectors;
\item[$(2)$] $T$ is super-recurrent.
\end{itemize}
\end{theorem}
%%%%%%%%%%%%%%%%%%%%%%%%%%%%%%%%%%%%%%%%%%%%%%%%%%%%%%%%%%%%%%%%%%%%%%%%%%%%%%%%%%%%%%%%%%%%%%%%%%%%%%%%%
\begin{proof}
$(1)\Rightarrow (2)$ :
Let $U$ be a nonempty open subset of $X$,
then there is a $T$-super-recurrent vector $x$ such that $x\in U$.
There exist a increasing sequence $(n_k)$ of positive integers and an sequence $(\lambda_{n_k})$ of complex  numbers such that
$\lambda_{n_k} T^{n_k}x\longrightarrow x$ as $k\longrightarrow+\infty$.
Since $U$ is open and $x\in U$,
it follows that there exist $\lambda\in\mathbb{C}$ and $n\in\mathbb{N}$ such that $\lambda T^n(U)\cap U\neq \emptyset,$
this means that $T$ is super-recurrent.\\
$(2)\Rightarrow (1)$ : For a fixed element $x\in X$ and a fixed strictly positive numbers $\varepsilon>0$, let
$$ B:=B(x,\varepsilon). $$
Since $T$ is super-recurrent,
there exist some positive integer $k_1$ and some number $\lambda_1$ such that
$\lambda_1 T^{-k_1}(B)\cap B\neq \emptyset.$
Let $x_1\in X$ such that $x_1\in \lambda_1 T^{-k_1}(B)\cap B$.
Since $T$ is continuous,
there exists $\varepsilon_1<\frac{1}{2}$ such that
$$ B_2:=B(x_1,\varepsilon_1)\subset  \lambda_1 T^{-k_1}(B)\cap B.$$
Again, since $T$ is super-recurrent,
there exist some $k_2\in\mathbb{N}$ and some $\lambda_2\in\mathbb{C}$ such that
$\lambda_2 T^{-k_2}(B_2)\cap B_2\neq \emptyset.$
Let $x_2\in X$ such that $x_2\in \lambda_2 T^{-k_2}(B_2)\cap B_2$.
By continuity of $T$, there exists $\varepsilon_2<\frac{1}{2^2}$ such that
$$ B_3:=B(x_2,\varepsilon_2)\subset \lambda_2 T^{-k_2}(B_2)\cap B_2. $$
Continuing inductively, we construct a sequence $(x_n)_{n\in\mathbb{N}}$ of elements of $X$,
a sequence $(\lambda_n)_{n\in\mathbb{N}}$ of complex numbers,
a strictly increasing sequence of positive integers $(k_n)_{n\in\mathbb{N}}$
and a sequence of positive real numbers $\varepsilon_n <\frac{1}{2^n}$,
such that
$$ B(x_n,\varepsilon_n)\subset B(x_{n-1},\varepsilon_{n-1})\hspace{0.2cm} \mbox{ and }\hspace{0.2cm}\lambda_n T^{n_k}\left( B(x_n,\varepsilon_n)\right)\subset B(x_{n-1},\varepsilon_{n-1}). $$
Since $X$ is a Banach space, then by Cantor's Theorem, there exists some vector $y\in X$ such that
\begin{equation}\label{eq11}
\bigcap_{n\in\mathbb{N}}B(x_n,\varepsilon_n)=\{y\}.
\end{equation}
Since $y\in B$, we need only to show that $y$ is $T$-super-recurrent.
By (\ref{eq11}), we have $y\in B(x_n,\varepsilon_n)$ for all $n$, which implies that
\begin{equation}\label{eq22}
\Vert x_n-y \Vert<\varepsilon_n.
\end{equation}
On the other hand, $\lambda_n T^{n_k}y\in B(x_n,\varepsilon_n)$. Indeed, we have $y\in B(x_{n+1},\varepsilon_{n+1})$. This implies that
$$\lambda_n T^{n_k} y\in \lambda_n T^{n_k}(B(x_{n+1},\varepsilon_{n+1}))\subset \lambda_n T^{n_k}(B(x_{n},\varepsilon_{n}))\subset B(x_{n},\varepsilon_{n}).$$
Hence,
\begin{equation}\label{eq32}
\Vert \lambda_n T^{n_k}y-x_n \Vert<\varepsilon_n.
\end{equation}
Now, by using (\ref{eq22}) and (\ref{eq32}) we conclude that
$$ \Vert \lambda_nT^{n_k}y-y \Vert\leq \Vert \lambda_nT^{n_k}y-x_n\Vert+ \Vert x_n-y \Vert<\frac{1}{2^{n-1}}. $$
Hence, $\lambda_nT^{n_k}y\longrightarrow y$, that is $y$ is a $T$-super-recurrent vector.
Hence each open ball of $X$ contains a $T$-super-recurrent vector.
Thus the set of all super-recurrent vectors for $T$ is dense in $X$.
\end{proof}
%%%%%%%%%%%%%%%%%%%%%%%%%%%%%%%%%%%%%%%%%%%%%%%%%%%%%%%%%%%%%%%%%%%%%%%%%%%%%%%%%%%%%%%%%%%%%%%%%%%%%%%%%

Theorem  \ref{th1} shows that any super-recurrent operator on a Banach space admits super-recurrent vectors. However, an operator may has super-recurrent vectors without being super-recurrent as we show in the following example.
%%%%%%%%%%%%%%%%%%%%%%%%%%%%%%%%%%%%%%%%%%%%%%%%%%%%%%%%%%%%%%%%%%%%%%%%%%%%%%%%%%%%%%%%%%%%%%%%%%%%%%%%%
\begin{example}
Let $X$ be a Banach space
and let $(e_i)_{i\in I}$ be a basis of $X$.
Let $i_0\in I$ and $\lambda\in\mathbb{C}$ a nonzero fixed  number.
We define an operator $T$ on $X$ by$:$
$$ Te_{i_0}=\lambda  e_{i_0} \hspace{0.3cm} \mbox{ and }\hspace{0.3cm} Te_i=0\mbox{, }\hspace{0.3cm}\mbox{ for all }i\in I\setminus\{i_0\}.$$
It is clear that $e_{i_0}$ is a $T$-super-recurrent vector for $T$.
However, $T$ itself is not super-recurrent since it is not of dense range and super-recurrent operators are of dense range by Remark \ref{rem dense}.
\end{example}
%%%%%%%%%%%%%%%%%%%%%%%%%%%%%%%%%%%%%%%%%%%%%%%%%%%%%%%%%%%%%%%%%%%%%%%%%%%%%%%%%%%%%%%%%%%%%%%%%%%%%%%%%
\begin{remark}
If $T$ is super-recurrent,
then $\lambda T$ is super-recurrent for all $\lambda\in\mathbb{C}^*.$
Moreover, $T$ and $\lambda T$ have the same super-recurrent vectors.
\end{remark}
%%%%%%%%%%%%%%%%%%%%%%%%%%%%%%%%%%%%%%%%%%%%%%%%%%%%%%%%%%%%%%%%%%%%%%%%%%%%%%%%%%%%%%%%%%%%%%%%%%%%%%%%%

The next theorem gives the relationship between the super-recurrence of an operator and its iterates.
%%%%%%%%%%%%%%%%%%%%%%%%%%%%%%%%%%%%%%%%%%%%%%%%%%%%%%%%%%%%%%%%%%%%%%%%%%%%%%%%%%%%%%%%%%%%%%%%%%%%%%%%%
\begin{theorem}\label{pr3}
Let $p$ be a nonzero positive integer.
Then, $T$ is super-recurrent if and only if $T^p$ is super-recurrent.
Moreover, $T$ and $T^p$ have the same super-recurrent vectors.
\end{theorem}
%%%%%%%%%%%%%%%%%%%%%%%%%%%%%%%%%%%%%%%%%%%%%%%%%%%%%%%%%%%%%%%%%%%%%%%%%%%%%%%%%%%%%%%%%%%%%%%%%%%%%%%%%
\begin{proof}
We will prove that $SRec(T)=SRec(T^p)$,
for that it is enough to show that $SRec(T)\subset SRec(T^p)$.
Let $x$ be a $T$-super-recurrent vector, then there exist a strictly increasing sequence $(k_n)_{n\in\mathbb{N}}$ of positive integers and a sequence $(\lambda_n)_{n\in\mathbb{N}}$ of complex numbers such that
$\lambda_n T^{k_n}x\longrightarrow x$ as $n\longrightarrow+\infty$.
Without loss of generality we may suppose that $k_n>p$ for all $n$.
Hence, for all $n$, there exist $\ell_n\in \mathbb{N}$ and $v_n\in\{0,\dots,p-1\}$ such that
$$ k_n=p\ell_n+v_n. $$
Since $(v_n)_{n}$ is bounded, there exists $v\in \{0,\dots,p-1\}$ and a subsequence of $(v_n)_{n}$ which converges to $v$.
Thus, $\lambda_{k_n}T^{p\ell_n+v}x\longrightarrow x$ for some subsequence of $(\ell_n)_{\in \mathbb{N}}$ and a subsequence $(\lambda_{k_n})_{\in \mathbb{N}}$
which we call them again $(\ell_n)_{\in \mathbb{N}}$ and $(\lambda_{k_n})_{\in \mathbb{N}}$.
Let $U$ be a nonempty open subset of $X$ such that $x\in U$.
Since $\lambda_{k_n}T^{p\ell_n+v}x\longrightarrow x$, there exists a positive integer $m_1:=\ell_{n_1}$ such that
$\lambda_{n_1}T^{pm_1+v}x\in U.$
We have
$$ \lambda_{k_n}\lambda_{n_1}T^{p(\ell_n+m_1)+2v}x=\lambda_{n_k}\lambda_{n_1}T^{p\ell_n+v}T^{pm_1+v}x\longrightarrow  \lambda_{n_1}T^{pm_1+v}x\in U.$$
Thus, we can find a positive integer $m_2:=m_1+\ell_{n_2}>m_1$ such that
$ \lambda_{n_1}\lambda_{n_2}T^{pm_2+2v}x\in U.$
Continuing inductively we can find a positive integer $m_p=m_{p-1}+\ell_{n_p}$ such that
$$ \lambda_{n_1}\dots\lambda_{n_p}T^{pm_p+pv}x\in U. $$
Put $\lambda=\lambda_{n_1}\dots\lambda_{n_p}$, then $\lambda (T^p)^{m_p+v}x\in U$,
which means that $x$ is $T^p$-super-recurrent.
Hence, $SRec(T)=SRec(T^p)$.
Now it suffices to use Theorem \ref{th1} to conclude the result.
\end{proof}
%%%%%%%%%%%%%%%%%%%%%%%%%%%%%%%%%%%%%%%%%%%%%%%%%%%%%%%%%%%%%%%%%%%%%%%%%%%%%%%%%%%%%%%%%%%%%%%%%%%%%%%%%
\section{Spectral Proprieties of Super-recurrent Operators}
%%%%%%%%%%%%%%%%%%%%%%%%%%%%%%%%%%%%%%%%%%%%%%%%%%%%%%%%%%%%%%%%%%%%%%%%%%%%%%%%%%%%%%%%%%%%%%%%%%%%%%%%%
In this section, we show that super-recurrent operators have some noteworthy spectral proprieties.

If $T$ is hypercyclic, then Kitai \cite{Kitai} showed that every component of the spectrum of $T$ must intersects the unit circle.
Later, N. S. Feldman, V. G. Miller, and T. L. Miller gave a similar result for the supercyclicity case.
They proved that if $T$ is supercyclic,
then there exists $R>0$ such that the circle $\{z\in\mathbb{C} \mbox{ : }\vert z \vert=R\}$,
called a supercyclicity circle for $T$, intersects each component of the spectrum of $T$,
see \cite[Theorem 1.24]{BM} or \cite{FMM}.
Recently, G. Costakis, A. Manoussos, and I. Parissis \cite{CMP} proved that the spectrum of recurrent operators share
the same propriety with hypercyclic operators by proven that
if $T$ is recurrent, then every component of the spectrum of $T$ intersects the unit circle.
Since super-recurrent operators "look like" supercyclic operators, it is expected that their spectrums share the same propriety.
This is the objective of the next theorem.
\begin{theorem}\label{t1}
Let $T$ be an operator acting on a complex Banach space $X.$
If $T$ is super-recurrent, then there exists $R>0$ such that each connected component of the spectrum of $T$ intersects the circle $\{z\in\mathbb{C} \mbox{ : }\vert z \vert=R\}$.
\end{theorem}
%%%%%%%%%%%%%%%%%%%%%%%%%%%%%%%%%%%%%%%%%%%%%%%%%%%%%%%%%%%%%%%%%%%%%%%%%%%%%%%%%%%%%%%%%%%%%%%%%%%%%%%%%
\begin{proof}
Assume that $T$ is super-recurrent.
We will produce by contradiction.
By \cite[Lemma 1.25]{BM}, there exist $R>0$ and $C_1$, $C_2$ two component of $\sigma(T)$ such that
$C_1\subset\mathbb{D}$ and $C_2\subset \mathbb{C}\setminus\overline{\mathbb{D}}$.
Without loss of generality, we may suppose that $R=1$. Indeed, this is since $T$ is super-recurrent if and only $R^{-1}T$ is.
By \cite[Lemma 1.21]{BM}, there exist $\sigma_1$ and $\sigma_2$, two closed and open sets of $\sigma(T)$ such that $C_1\subset \sigma_1\subset \mathbb{D}$ and $C_2\subset \sigma_2\subset\mathbb{C}\setminus\overline{\mathbb{D}}$.
Set $\sigma_3=\sigma(T)\setminus(\sigma_1\cup\sigma_2)$.
We have then $\sigma(T)=\sigma_1\cup\sigma_2\cup\sigma_3$ and the sets $\sigma_i$ are closed and pairwise
disjoint.
By Reisz decomposition theorem there exist $X_1$, $X_2$, $X_3$ and $T_1$, $T_2$, $T_3$
such that $X=X_1\oplus X_2\oplus X_2$ and $T=T_1\oplus T_2\oplus T_3$,
where each $X_i$ is a $T$-invariant subspace, $T_i=T_{/X_i}$ and $\sigma_i=\sigma(T_i)$.
Let $x\in X_1$ and $y\in X_2$.
By Theorem \ref{tt} , there exist $(\lambda_k)\subset\mathbb{C}$, $(n_k)\subset\mathbb{N}$, $(x_k)\subset X_1$ and $(y_k)\subset X_2$ such that
$$ x_k\longrightarrow x\mbox{, } \hspace{0.2cm}y_k\longrightarrow y \mbox{, } \hspace{0.2cm}
 \lambda_k T_1^{n_k}x_k\longrightarrow x  \hspace{0.2cm}\mbox{ and } \hspace{0.2cm} \lambda_k T_2^{n_k}y_k\longrightarrow y.$$
By \cite[Lemma 1.20]{BM}, the last assertion implies that $(\vert \lambda_k\vert)$ converges into $0$ and $+\infty$, which is a contradiction.
\end{proof}
%%%%%%%%%%%%%%%%%%%%%%%%%%%%%%%%%%%%%%%%%%%%%%%%%%%%%%%%%%%%%%%%%%%%%%%%%%%%%%%%%%%%%%%%%%%%%%%%%%%%%%%%%

The adjoint Banach operator of a hypercyclic operator cannot have eigenvalue. This means that $\sigma_p(T^*)=\emptyset$, see \cite[Proposition 1.7]{BM}.
Unlike the hypercyclicity case, the adjoint of a supercyclic operator $T$ can have an eigenvalue but not more then one.
This means that either we have $\sigma_p(T^*)=\emptyset$ or there exists $\lambda$ such that
$\sigma_p(T^*)=\{\lambda\}.$
For the recurrent operators, it is expected that they have the same result as hypercyclic operators, but
this is not the case, see \cite[Example 2.13 and Remark 2.15]{CMP}.
So the Banach adjoint operator of a recurrent operator may has eigenvalue.
However, no one of those eigenvalue can be outside of the unit circle.
This means that $\sigma_p(T^*)\subset \mathbb{T}$, where $\mathbb{T}$ the unit circle.
Since recurrent operators are super-recurrent, it follows that some super-recurrent operators may have eigenvalue.
However, all those eigenvalues lie in a circle of form $\{z\in\mathbb{C} \mbox{ : }\vert z \vert=R\}$, where $R>0$.
This is the content of the next result.
%%%%%%%%%%%%%%%%%%%%%%%%%%%%%%%%%%%%%%%%%%%%%%%%%%%%%%%%%%%%%%%%%%%%%%%%%%%%%%%%%%%%%%%%%%%%%%%%%%%%%%%%%
\begin{theorem}\label{tt1}
The eigenvalues of the adjoint operator of a super-recurrent operator have the same argument.
That is,
if $T$ is super-recurrent, then there exists $R>0$ such that $\sigma_p(T^{*})\subset \{z\in\mathbb{C} \mbox{ : }\vert z \vert=R\}$. In particular, for all $\lambda\in\mathbb{C}\setminus\{z\in\mathbb{C} \mbox{ : }\vert z \vert=R\}$ the operator $T-\lambda I$ has dense range.
\end{theorem}
%%%%%%%%%%%%%%%%%%%%%%%%%%%%%%%%%%%%%%%%%%%%%%%%%%%%%%%%%%%%%%%%%%%%%%%%%%%%%%%%%%%%%%%%%%%%%%%%%%%%%%%%%
\begin{proof}
Assume that there exist $\lambda$, $\mu\in \sigma_p(T^{*})$ such that $\vert\mu\vert<\vert\lambda\vert$ and let $m$ be a nonzero real number such that $\vert\mu\vert<m<\vert\lambda\vert$.
Since $\lambda$, $\mu\in \sigma_p(T^{*})$, there exist $x^{*}$, $y^{*}\in X^{*}$ such that $T^{*} x^{*}=\lambda x^{*}$ and $T^{*} y ^{*}=\mu y^{*}$.
This implies that $x^{*}(T^nz)=\lambda^n x^{*}(z)$ and $y^{*}(T^nz)=\mu^n y^{*}(z)$ for all $z\in X$.
Since $T$ is super-recurrent if and only $\frac{1}{m}T$ is, let $z_0\in SRec(\frac{1}{m}T).$
By Baire Category Theorem we may suppose that $x^{*}(z_0)\neq0$ and $y^{*}(z_0)\neq0$.
Since $z_0$ is a super-recurrent vector for $\frac{1}{m}T$,
it follows that there exist $(\beta_{k})\subset\mathbb{C}$ and $(n_k)\subset \mathbb{N}$
such that $\beta_k \frac{1}{m^{n_k}}T^{n_k}z_0\longrightarrow z_0$ as $k\longrightarrow\infty$.
Since $x^{*}$ and $y^{*}$ are continuous, we deduce that
$$\beta_k \left( \frac{\lambda}{m} \right)^{n_k}x^{*}(z_0)\longrightarrow x^{*}(z_0)
\hspace{0.2cm}\mbox{ and } \hspace{0.2cm}
\beta_k \left( \frac{\mu}{m} \right)^{n_k}y^{*}(z_0)\longrightarrow y^{*}(z_0).$$
Using that $x^{*}(z_0)\neq0$ and $y^{*}(z_0)\neq0$ we conclude that
$\beta_k \left( \frac{\lambda}{m} \right)^{n_k}\longrightarrow1$
and
$\beta_k \left( \frac{\mu}{m} \right)^{n_k}\longrightarrow1$
Hence $\vert \beta_k\vert \longrightarrow 0$ and $\vert \beta_k\vert \longrightarrow \infty$, which is a contradiction.
\end{proof}
%%%%%%%%%%%%%%%%%%%%%%%%%%%%%%%%%%%%%%%%%%%%%%%%%%%%%%%%%%%%%%%%%%%%%%%%%%%%%%%%%%%%%%%%%%%%%%%%%%%%%%%%%
\begin{remark}
If $T$ is supercyclic, then $T$ is super-recurrent, but either $\sigma_p(T^{*})=\emptyset$ or $\sigma_p(T^{*})=\{\lambda\}$ for some nonzero number $\lambda$.
However, there exist several super-recurrent operators such that $Card(\sigma_p(T^{*}))>1$.
Indeed, let $(\lambda_n)_{n\in\mathbb{N}}$ be a sequence of nonzero complex numbers of the same argument.
Define in $\ell^2(\mathbb{N})$ an operator $T$ by
$$T(x_1,x_2,\dots)=(\lambda x_1,\lambda_2  x_2,\dots).$$
Then $T$ is a super-recurrent operator.
It's easy to check that $(\overline{\lambda_n})_{n\in\mathbb{N}}\subset\sigma_p(T^{*})$ and hence $\sigma_p(T^{*})$ is an infinite set.
\end{remark}
%%%%%%%%%%%%%%%%%%%%%%%%%%%%%%%%%%%%%%%%%%%%%%%%%%%%%%%%%%%%%%%%%%%%%%%%%%%%%%%%%%%%%%%%%%%%%%%%%%%%%%%%%

We already know that if $T$ is supercyclic, then either $\sigma_p(T^{*})=\emptyset$ or $\sigma_p(T^{*})=\{\lambda\}$ for some nonzero number $\lambda$.
Moreover, in the latter case, one can find a $T$-invariant hyperplane $X_0\subset X$ such that the operator
$T_0:=T_{/X_0}$ is hypercyclic on $X_0$, see \cite[Proposition 1.26]{BM}.
In the next theorem, we prove that the same relation still true between recurrent and super-recurrent operators.
%%%%%%%%%%%%%%%%%%%%%%%%%%%%%%%%%%%%%%%%%%%%%%%%%%%%%%%%%%%%%%%%%%%%%%%%%%%%%%%%%%%%%%%%%%%%%%%%%%%%%%%%%
\begin{theorem}
Let $X$ be a Banach space with dim$(X)>1$. Let $T$ be a super-recurrent operator acting on $X$.
Then for all $\lambda\in\sigma_p(T^{*})$, there exists a $($closed$)$ $T$-invariant hyperplane $X_0\subset X$ such that $T_0:=\lambda^{-1}T_{/X_0}$ is recurrent on $X_0.$
\end{theorem}
%%%%%%%%%%%%%%%%%%%%%%%%%%%%%%%%%%%%%%%%%%%%%%%%%%%%%%%%%%%%%%%%%%%%%%%%%%%%%%%%%%%%%%%%%%%%%%%%%%%%%%%%%
\begin{proof}
First note that $\lambda\neq0$ for every $\lambda\in \sigma_p(T^{*})$ since a super-recurrent operator has dense range.

Since $T$ is super-recurrent if and only if $aT$ is super-recurrent for every $a\neq0$, we may assume, without loss of generality, that $\lambda=1.$
Choose $x_0^{*}\in X^{*}\setminus\{0\}$ such that $T^{*}x_0^{*}=x_0^{*}$ and let $X_0=Ker(x_0^{*})$.
Since $x_0^{*}$ is an eigenvector of $T^{*}$, it follows that $X_0$ is a $T$-invariant hyperplane of $X$.
We can consider then $T_0:=T_{/X_0}$.
In the following, we will prove that $T_0$ is a recurrent operator on $X_0$.

With a slight abuse of notation, we may write $X=\mathbb{C}\oplus X_0$ and since $T^{*}x_0^{*}=x_0^{*}$, let $T(1\oplus 0)=1\oplus y$ for some $y\in X_0$.
It follows then that $T(1\oplus z)=1\oplus (y+T_0(z)$ for all $z\in X_0$.
By straightforward induction, we have
$$ T^n(1\oplus z)=1\oplus(y+T_0(y)+\dots+T_0^{n-1}(y)+T_0^n(z)) $$
for all $z\in X_0.$

Note that $T_0-I$ has dense range. Indeed, assume that $\overline{(T_0-I)(X_0)}\neq X_0$ and without loss of generality we may suppose that $y\notin \overline{(T_0-I)(X_0)}.$
By the Hahn-Banach theorem, there exists $k^{*}\in X_0^{*}$ such that $k^{*}(y)\neq 0$ and $k^{*}(T^nz)=k^{*}(z)$ for every $z\in X_0$.
Choose a super-recurrent vector for $T$ of the form $1\oplus x_0$.
Hence there exist $(\mu_k)\subset\mathbb{C}$ and a strictly increasing sequence $(n_k)\subset \mathbb{N}$
such that
$\mu_k T^{n_k}(1\oplus x_0)\longrightarrow 1\oplus x_0$ as $k\longrightarrow\infty.$
Thus
$$ \mu_k( 1\oplus(y+T_0(y)+\dots+T_0^{n-1}(y)+T_0^n(x_0)) )\longrightarrow 1\oplus x_0. $$
This implies that $\mu_k\longrightarrow1$ and $y+T_0(y)+\dots+T_0^{n_k-1}(y)+T_0^{n_k}(x_0))\longrightarrow x_0$.
Since $k^{*}$ is continuous and $k^{*}(y)\neq0$, it follows that $n_k-1\longrightarrow0$, which is a contradiction.

Since $T$ is super-recurrent, there exist a subset $A$ of $\mathbb{C}$ and a subset $B$ of $X_0$ such that. $SRec(T)=A\oplus B$ such that $\overline{A}=\mathbb{C}$ and $\overline{B}=X_0$.

Finally, let $x$ be an element of $B$. By the same method applied to $x_0$, we have
$$y+T_0(y)+\dots+T_0^{n-1}(y)+T_0^n(x))\longrightarrow x.$$
Applying $(T_0-I)$, we get
$$ T^{n_k}(y+(T_0-I)x)\longrightarrow (y+(T_0-I)x.$$
This implies that $(y+(T_0-I)x\in Rec(T_0)$.
Since $(T_0-I)$ has dense range, we conclude that $T_0$ is recurrent on $X_0.$
\end{proof}
%%%%%%%%%%%%%%%%%%%%%%%%%%%%%%%%%%%%%%%%%%%%%%%%%%%%%%%%%%%%%%%%%%%%%%%%%%%%%%%%%%%%%%%%%%%%%%%%%%%%%%%%%
The Purpose of the following proposition is to show that a large supply of eigenvectors corresponding
to eigenvalues with same argument implies that the operator is super-recurrent.
%%%%%%%%%%%%%%%%%%%%%%%%%%%%%%%%%%%%%%%%%%%%%%%%%%%%%%%%%%%%%%%%%%%%%%%%%%%%%%%%%%%%%%%%%%%%%%%%%%%%%%%%%
\begin{proposition}
Let $T$ be an operator acting on $X$.
If there exists $R>0$ such that the space generated by
$$ X_0:=\{x\in X \mbox{ : } Tx=\lambda x \mbox{ for some }\lambda\in\{\vert \lambda\vert=R\}\} $$
is dense in $X$, then $T$ is super-recurrent.
\end{proposition}
%%%%%%%%%%%%%%%%%%%%%%%%%%%%%%%%%%%%%%%%%%%%%%%%%%%%%%%%%%%%%%%%%%%%%%%%%%%%%%%%%%%%%%%%%%%%%%%%%%%%%%%%%
\begin{proof}
Let $\sum_{i=1}^n a_i x_i\in$ span $\{X_0\}$, where $Tx_i=\lambda_i x_i$, for certain $a_i$, $\lambda_i\in\mathbb{C}$
with $\vert \lambda_i\vert=R$ for $i=1,\dots,n.$
Since each $R^{-1}\lambda_i$ is in the unite circle,
it follows that there exists a strictly increasing sequence $(n_k)$ such that $\left( R^{-1}\lambda_i\right)^{n_k}\longrightarrow 1$ as $k\longrightarrow\infty.$
Hence 
$$ R^{-n_k}T^{n_k}\left( \sum_{i=1}^n a_i x_i \right)=\sum_{i=1}^n a_i R^{-n_k} \lambda_i x_i\longrightarrow  \sum_{i=1}^n a_i x_i$$
 as $k\longrightarrow\infty.$
 This means that span$\{X_0\}\subset SRec(T)$.
 Since span$\{X_0\}$ is dense in $X$, it follows that $T$ is super-recurrent.
\end{proof}
%%%%%%%%%%%%%%%%%%%%%%%%%%%%%%%%%%%%%%%%%%%%%%%%%%%%%%%%%%%%%%%%%%%%%%%%%%%%%%%%%%%%%%%%%%%%%%%%%%%%%%%%%

\end{document}